\documentclass{article}

\usepackage{amsmath}
\usepackage{amsthm}
\usepackage{amsfonts}
\usepackage{mathrsfs}
\usepackage{amssymb}
\usepackage{bbm}
\usepackage{yhmath}

\usepackage{tikz}
\usetikzlibrary{cd,matrix,arrows,decorations.pathmorphing}

\usepackage{enumerate}

\usepackage{hyperref}
\hypersetup{colorlinks=true, linkcolor=black, filecolor=black, urlcolor=black, citecolor=black}

\theoremstyle{definition}
\newtheorem{theorem}{Theorem}[subsection]

\newtheorem{proposition}[theorem]{Proposition}
\newtheorem{corollary}[theorem]{Corollary}

\def\calC{\mathcal{C}}

\def\calG{\mathcal{G}}

\def\calM{\mathcal{M}}

\def\calO{\mathcal{O}}

\def\frg{\mathfrak{g}}

\def\frm{\mathfrak{m}}
\def\frn{\mathfrak{n}}

\def\frp{\mathfrak{p}}

\def\ffrb{\mathfrak{B}}

\def\ffrm{\mathfrak{M}}

\def\ffrx{\mathfrak{X}}

\def\bbA{\mathbb{A}}

\def\bbF{\mathbb{F}}
\def\bbG{\mathbb{G}}

\def\bbP{\mathbb{P}}
\def\bbQ{\mathbb{Q}}

\def\bbZ{\mathbb{Z}}

\DeclareFontFamily{U}{wncy}{}
\DeclareFontShape{U}{wncy}{m}{n}{<->wncyr10}{}
\DeclareSymbolFont{mcy}{U}{wncy}{m}{n}
\DeclareMathSymbol{\Sha}{\mathord}{mcy}{"58} 

\def\bs{\backslash}

\DeclareMathOperator{\ox}{\otimes}

\def\comp{\circ}
\def\inj{\hookrightarrow}

\def\isom{\cong}
\def\aisom{\xrightarrow{\sim}}
\def\ot{\leftarrow}

\newcommand{\dlim}{\varinjlim}
\newcommand{\ilim}{\varprojlim}

\DeclareMathOperator{\Tot}{Tot}

\DeclareMathOperator{\GL}{GL}

\DeclareMathOperator{\Lie}{Lie}

\DeclareMathOperator{\gl}{\mathfrak{gl}}

\def\bb1{\mathbbm{1}}

\DeclareMathOperator{\DR}{DR}

\DeclareMathOperator{\HT}{HT}
\DeclareMathOperator{\dR}{dR}

\DeclareMathOperator{\Spf}{Spf}

\def\LT{\text{LT}}

\def\lan{\text{la}}
\def\sm{\text{sm}}

\def\hat{\widehat}
\def\bar{\overline}
\def\tilde{\widetilde}

\def\-{\text{-}}

\def\-{\text{-}}

\def\fl{\mathscr{F}\ell}

\def\GM{\text{GM}}

\def\Nilp{\text{Nilp}}

\def\Dr{\text{Dr}}

\def\ov{\stackrel}

\usepackage{geometry}
\begin{document}
\title{de Rham cohomology of Lubin-Tate spaces and Drinfeld spaces}
\author{Benchao Su}
\date{}
\maketitle
\tableofcontents

\begin{abstract}
Let $d\ge 1$ be an integer. We use the methods \cite{pan_locally_2022}, \cite{pan_locally_2022-1} introduced by Lue Pan to prove that the compactly supported cohomology of Lubin-Tate towers and Drinfeld towers are isomorphic, as $\GL_{d+1}(L)\times D_{L,\frac{1}{d+1}}^\times$-modules.
\end{abstract}

\subsection{Introduction}
Let $L$ be a finite extension of $\bbQ_p$, and let $d\ge 1$ be an integer. Let $G=\GL_{d+1}(L)$, and let $\check{G} = D_{L, \frac{1}{d+1}}^\times$ be the group of invertible elements in a division algebra over $L$ with invariant $\frac{1}{d+1}$. Let $C$ be the $p$-adic completion of an algebraic closure of $L$. The foundational works \cite{MR238854,drinfeld_coverings_1976,rapoport_period_1996} introduced two important towers of rigid analytic spaces $\{\calM_{\LT,n}\}_n$, $\{\calM_{\Dr,n}\}_n$ over $C$, namely the Lubin-Tate towers and the Drinfeld towers. They carry natural $G\times\check{G}$-actions and the action of the Weil group of $L$. The non-abelian Lubin-Tate theory \cite{MR1044827,harris_geometry_2001,MR2308851} shows that the $\ell$-adic cohomology of these towers encodes the local Langlands correspondence and the local Jacquet-Langlands correspondence. In the $p$-adic case, several recent works \cite{CDN20,junger2022cohomologierhamdurevetementLT,junger2022cohomologierhamdurevetementDr} have shown that in certain cases the compactly supported de Rham cohomology of these towers also encodes the local Langlands correspondence and the local Jacquet-Langlands correspondence. The main theorem of this note is as follows:
\begin{theorem}\label{thm:main}
For any $i\ge 0$, there exists a $C$-linear $G\times \check G$-equivariant isomorphism of the compactly supported de Rham cohomology groups
\begin{align*}
    H^i_{\dR,c}(\calM_{\LT,\infty})\isom H^i_{\dR,c}(\calM_{\Dr,\infty}).
\end{align*}
\end{theorem}
In the case where the dimension $d=1$, this was proved in \cite[Théorème 4.1]{CDN20}. In general, this was proved in \cite[Theorem 5.2.2]{dospinescu2024jacquetlanglandsfunctorpadiclocally} and by Bosco-Dospinescu-Niziol (flip-flopping). We provide a relatively elementary and explicit approach to these results.

First, from \cite{scholze_moduli_2013} we know that by taking the limit over the levels of these two towers, we obtain two perfectoid spaces $\calM_{\LT,\infty}$, $\calM_{\Dr,\infty}$ over $C$, which are $G\times \check{G}$-equivariantly isomorphic:
\begin{align*}
    \calM_{\LT,\infty}\isom \calM_{\Dr,\infty}.
\end{align*}
Let $\calO_{\calM_{\LT,\infty}}$ be the completed structure sheaf on $\calM_{\LT,\infty}$, and let $\calO_{\calM_{\LT,\infty}}^{\sm}\subset \calO_{\calM_{\LT,\infty}}$ be the subsheaf consisting of sections which are fixed by a sufficiently small open compact subgroup of $G$. Then we can associate a de Rham complex 
\begin{align*}
    \DR_{\LT}:=[\calO_{\calM_{\LT,\infty}}^{\sm}\ov{d}\to \Omega_{\calM_{\LT,\infty}}^{1,\sm}\ov{d}\to \cdots\ov{d}\to \Omega_{\calM_{\LT,\infty}}^{d,\sm}]
\end{align*} 
given by taking differentials along the Lubin-Tate tower of finite levels, where $\Omega_{\calM_{\LT,\infty}}^{k,\sm}:=\dlim_n\pi_n^{-1}\Omega_{\calM_{\LT,n}}^k$ with $\pi_n:\calM_{\LT,\infty}\to \calM_{\LT,n}$ for $k=1,...,d$. The compactly supported de Rham cohomology groups $H^i_{\dR,c}(\calM_{\LT,\infty})$ are defined (in a suitable sense) as the compactly supported cohomology groups of the complex $\DR_{\LT}$. Similarly, on the Drinfeld side we can define the de Rham complex $\DR_{\Dr}$ and the compactly supported de Rham cohomology groups $H^i_{\dR,c}(\calM_{\Dr,\infty})$ for the Drinfeld tower. The isomorphism of Theorem \ref{thm:main} then can be given by the isomorphism of the de Rham complexes $\DR_{\LT}$ and $\DR_{\Dr}$.

However, it is difficult to directly deduce the isomorphism of de Rham cohomology groups from the isomorphism on the completed structure sheaves $\calO_{\calM_{\LT,\infty}}\isom \calO_{\calM_{\Dr,\infty}}$, since the naive definition of de Rham cohomology for perfectoid spaces is problematic. It turns out that on a smaller subsheaf of $\calO_{\calM_{\LT,\infty}}$ containing $\calO_{\calM_{\LT,\infty}}^{\sm}$ (which we will define later), we can lift the differential operators on $\calO_{\calM_{\LT,\infty}}^{\sm}$ to this slightly larger sheaf. This allows us to use the isomorphism $\calO_{\calM_{\LT,\infty}}\isom \calO_{\calM_{\Dr,\infty}}$ and directly compare these two de Rham complexes.

Motivated by the works \cite{pan_locally_2022}, \cite{pan_locally_2022-1}, \cite{camargo_geometric_2022}, \cite{camargo2024locallyanalyticcompletedcohomology}, we are led to study the subsheaf consisting of locally $L$-analytic sections (for the $G$-action) $\calO_{\calM_{\LT,\infty}}^{L\-\lan}\subset \calO_{\calM_{\LT,\infty}}$. By geometric Sen theory \cite{camargo_geometric_2022}, there exists an action of the Lie algebroid $\pi^{-1}\frm^0$ on $\calO_{\calM_{\LT,\infty}}^{L\-\lan}$, where $\pi:\calM_{\LT,\infty}\to \fl$ is the Hodge-Tate period map \cite{scholze_moduli_2013}, and $\frm^0\isom \frp^0/\frn^0$ with $\frp^0$ being the universal parabolic subalgebra on $\fl$ and $\frn^0$ being the universal nilpotent algebra on $\fl$. Let $\calO_{\calM_{\LT,\infty}}^{L\-\lan,\frm^0=0}\subset \calO_{\calM_{\LT,\infty}}^{L\-\lan}$ be the subsheaf consisting of sections killed by the $\pi^{-1}\frm^0$-action. Then we study the local structure of $\calO_{\calM_{\LT,\infty}}^{L\-\lan,\frm^0=0}$ in detail. Moreover, although the sheaf we consider on the Drinfeld side $\calO_{\calM_{\Dr,\infty}}^{\sm}$ consisting of sections that are smooth under the $\check{G}$-action, these sections are locally $L$-analytic under the $G$-action. Then one can show that the natural map 
\begin{align*}
    \calO_{\calM_{\LT,\infty}}^{\sm}\ox_C\calO_{\calM_{\Dr,\infty}}^{\sm}\to \calO_{\calM_{\LT,\infty}}^{L\-\lan,\frm^0=0}
\end{align*}
has dense image (Corollary \ref{cor:OOcheckdense}). We also define two differential operators on $\calO_{\calM_{\LT,\infty}}^{L\-\lan,\frm^0=0}$ in the same spirit as in \cite{pan_locally_2022-1}, which is essentially given by taking differentials along the Lubin-Tate spaces of finite level or the Drinfeld spaces of finite level. In order to compute the cohomology of these differential operators, we do the same construction on the Drinfeld side. In particular, there is also a subsheaf $\calO_{\calM_{\Dr,\infty}}^{L\-\lan,\check\frm^0=0}$ on $\calM_{\Dr,\infty}$ consisting of locally $L$-analytic sections for the $\check G$-action, and killed by the horizontal action of $\check{\pi}^{-1}\check\frm^0$ with $\check{\pi}:\calM_{\Dr,\infty}\to \check{\fl}$ the Hodge-Tate period map on the Drinfeld side. We check directly that under the isomorphism $\calO_{\calM_{\LT,\infty}}\isom \calO_{\calM_{\Dr,\infty}}$, there is a $G\times \check G$-equivariant isomorphism (Theorem \ref{thm:isom})
\begin{align*}
    \calO_{\calM_{\LT,\infty}}^{L\-\lan,\frm^0=0}\isom \calO_{\calM_{\Dr,\infty}}^{L\-\lan,\check\frm^0=0}.
\end{align*}
We also identify the differential operators on these sheaves appropriately (Theorem \ref{thm:ddbar}). This allows us to compute the cohomology of these differential operators. From this computation, we can deduce our main result.

\subsubsection*{Acknowledgements}
We sincerely thank Arnaud Vanhaecke, Liang Xiao, Tian Qiu, Yiwen Ding, Yuanyang Jiang, Zhenghui Li, Zhixiang Wu for helpful discussions and comments. We also thank Gabriel Dospinescu for his encouragements. We are especially grateful to my PhD supervisor, Yiwen Ding, for his invaluable guidance in improving my writing skills.

\subsubsection*{Notation and conventions}
Fix a prime number $p$. Let $L$ be a finite extension of $\bbQ_p$ (this is the base field), with ring of integers $\calO_L$. We fix a uniformizer $\varpi$ of $L$, and $|\calO_L/\varpi\calO_L|=q$. Fix an algebraic closure $\bar L$ of $L$, and let $C$ be the $p$-adic completion of $\bar L$. Let $\breve{L}$ be the $p$-adic completion of the maximal unramified extension of $L$ inside $C$.

\subsection{The flag variety}
Let $d\ge 1$ be an integer. Let $\bbP^d$ be the $d$-dimensional projective space over $L$. We identify it with $\bbP^d\isom (\bbA^{d+1}\bs\{0\})/\bbG_m$ where $\bbG_m$ acts on $\bbA^{d+1}\bs\{0\}$ by dilation. Let $o\in\bbP^d(L)$ be a distinguished $L$-rational point in $\bbP^d$. After choosing coordinates on $\bbP^d$, we may write $o=[0:0:\cdots:0:1]$. Let $\GL_{d+1,L}$ be the reductive group over $L$ consisting of invertible $(d+1)\times (d+1)$ matrices. It acts on $\bbP^d$ from the right:
\begin{align*}
    [z_0:z_1:\cdots:z_{d}]g:=[(z_0,z_1,...,z_d)g].
\end{align*}
where $v\mapsto [v]$ means passing to the $\bbG_m$-orbit of $v\in\bbA^{d+1}$. Under this right $\GL_{d+1,L}$-action, $\bbP^d$ is identified with $\bbP^d\isom P_{d,1}\bs \GL_{d+1,L}$, where $P_{d,1}\subset \GL_{d+1,L}$ is the subgroup consisting of matrices of the form $\left( \begin{matrix} A&\alpha\\0&d \end{matrix} \right)$, with $A\in \GL_{d,L}$, $\alpha\in \bbA^d_L$ and $d\in \bbG_{m,L}$. By transposition $g\mapsto g^t$, we get a left action of $\GL_{d+1,L}$ on $\bbP^d$:
\begin{align*}
    g[z_0:z_1:\cdots:z_d]^t:=[g(z_0,z_1,...,z_d)^t].
\end{align*}
This left action identifies $\bbP^d\isom \GL_{d+1,L}/\bar P_{d,1}$, where $\bar P_{d,1}$ is the opposite parabolic of $\bbP_{d,1}$. 

There are some equivariant vector bundles on $\bbP^d$ associated to some Lie subalgebras of $\frg$, where $\frg=\gl_{d+1}(L)$ is the Lie algebra of $\GL_{d+1,L}$. Let $\frp$ be the Lie algebra of $P_{d,1}$, with $\frn$ the nilpotent radical of $\frp$ and $\frm$ the Levi complement. Let $\bar\frp$ be the opposite parabolic of $\frp$, with $\bar\frn$ be the nilpotent radical of $\bar\frp$. Let $\fl$ be the adic space associated to $\bbP^d_L\times_{L,\iota}C$. Let $\frg^0:=\calO_{\fl}\ox \frg$ and we view it as an Lie algebroid over $\fl$ such that $[f\ox X,g\ox Y]=fX(g)\ox Y-gY(f)\ox X+fg\ox [X,Y]$ where $f,g\in \calO_{\fl}$ and $X,Y\in \frg$. Let $\frp^0\subset\frg^0$ be the universal parabolic subalgebra on $\fl$, with $\frn^0$ the universal nilpotent subalgebra. More precisely,
\begin{itemize}
    \item $\frp^0$ is the subbundle of $\frg^0$ consisting of vector fields $\ffrx\in\frg^0$ such that $\ffrx_x\in \frp_x$, where $\frp_x$ is the parabolic subalgebra defined by the point $x$. If we choose a lift $\tilde{x}$ of $x$ into $\GL_{d+1}(L)$, then $\frp_x=\tilde{x}^{-1}\frp \tilde{x}$.
    \item Similarly, $\frn^0$ is the subbundle of $\frg^0$ consisting of vector fields $\ffrx\in\frg^0$ such that $\ffrx_x\in \frn_x$, where $\frn_x$ is the nilradical of the parabolic subalgebra defined by the point $x$.
\end{itemize}
We note that the Lie algebroids $\frp^0\supset\frn^0$ are ideals of $\frg^0$.

Write $G=\GL_{d+1}(L)$. Let $D=D_{L,\frac{1}{d+1}}$ be the division algebra over $L$ of invariant $\frac{1}{d+1}$. The set of invertible elements in $D$ defines a reductive group $\check\bbG$ over $L$. We put $\check{G}:=\check{\bbG}(L)$ and view it as a locally $L$-analytic group. Let $\check\frg$ be the Lie algebra of $\check\bbG$. Fix an embedding $\iota:L\inj C$, and we fix an isomorphism 
\begin{align}\label{ggcheckisom}
    \frg\ox_{L,\iota}C\isom \check\frg\ox_{L,\iota}C.\tag{$*$}
\end{align}
Let $\check{X}$ be the Brauer-Severi variety over $L$ associated to $D$, and we let $\check{\fl}$ be the adic space associated to $\check X\times_{L,\iota}C$. Under the fixed isomorphism (\ref{ggcheckisom}), the parabolic subgroup $\frp\ox_{L,\iota}C\subset \frg\ox_{L,\iota}C$ also defines a parabolic subgroup $\check{\frp}$ of $\check{\frg}\ox_{L,\iota}C$. Similarly we get Lie algebroids $\check{\frg}^0\supset\check{\frp}^0\supset\check\frn^0$ on $\check{\fl}$.

\subsection{The Lubin-Tate space at infinite level}

Let $H_0$ be a one-dimensional compatible $\calO_L$-module over $\bar\bbF_q$ of $\calO_L$-height $d+1$, which is unique up to quasi-isogeny. Let $\Nilp_{\calO_{\breve{L}}}$ be the category of $\calO_{\breve{L}}$-algebras such that $\varpi$ is nilpotent. Let $\ffrm_{\LT,0}$ be the functor from $\Nilp_{\calO_{\breve{L}}}$ to the category of sets, assigning an $\calO_{\breve{L}}$-algebra $R$ to the set of equivalence classes of pairs $(G,\rho)$, where $G$ is a compatible $p$-divisible $\calO_L$-module over $R$ and $\rho:H_0\ox_{\bar\bbF_q}R/\varpi\dasharrow G\ox_R R/\varpi$ is an $\calO_L$-equivariant quasi-isogeny. Here the property of being compatible means that the natural $\calO_L$-action on $\Lie(G)$ induced by the $\calO_L$-action on $G$, coincides with the $\calO_L$-action induced by the natural action of $R$ on $\Lie(G)$.  By \cite{rapoport_period_1996}, this functor is represented by a formal scheme over $\Spf(\calO_{\breve{L}})$. We denote by $\calM_{\LT,0}$ the base change of the adic generic fiber of $\ffrm_{\LT,0}$ from $\breve L$ to $C$.

Let $(\calG,\rho)$ be the universal deformation of $H_0$ on $\ffrm_{\LT,0}$. The $\varpi$-adic Tate module of $\calG$ defines a $\bbZ_p$-local system on the \'etale site of $\calM_{\LT,0}$, which we denoted it by $V_{\LT}$. For any $n\ge 0$, by considering trivializations of $V_{\LT}/\varpi^nV_{\LT}$, we get an \'etale Galois covering $\calM_{\LT,n}$ of $\calM_{\LT,0}$ with Galois group $\GL_{d+1}(\calO_L/\varpi^n\calO_L)$. By \cite{scholze_moduli_2013}, there exists a perfectoid space $\calM_{\LT,\infty}$ over $C$ such that 
\[
    \calM_{\LT,\infty}\sim\ilim_n\calM_{\LT,n}.
\]
There is a natural continuous right action of $\GL_{d+1}(\calO_L)$ on $\calM_{\LT,\infty}$, which can be extended to $G=\GL_{d+1}(L)$. Also $\check{G}$ acts on each layer $\calM_{\LT,n}$, $n\ge 0$. According to the height of the quasi-isogeny in the moduli problem, the space $\calM_{\LT,\infty}$ decomposes as a disjoint union $\calM_{\LT,\infty}\isom \sqcup_{i\in\bbZ}\calM_{\LT,\infty}^{(i)}$. Let $G^0=\{g\in G:v_p(\det g)=0\}$, then this decomposition is $G^0$-equivariant. The Hodge filtration of the Dieudonn\'e module of the universal $p$-divisible group $(\calG,\rho)$ defines the Gross-Hopkins period map \cite{hopkins_equivariant_1994}
\begin{align*}
    \pi_{\LT,\GM}:\calM_{\LT,\infty}\to \check{\fl}.
\end{align*}
This map factors as the composition of the projection $\calM_{\LT,\infty}\to \calM_{\LT,0}$ with $\pi_{\LT,\GM,0}:\calM_{\LT,0}\to \check{\fl}$.
The map $\pi_{\LT,\GM,0}$ is an \'etale map \cite[Proposition 7.2]{de_jong_etale_1995}, \cite{achinger_variants_2022}, and admits local sections. Besides, the Hodge filtration of the rational Tate module of the universal $p$-divisible group $(\calG,\rho)$ defines the Hodge-Tate period map \cite{scholze_moduli_2013}
\begin{align*}
    \pi_{\LT,\HT}:\calM_{\LT,\infty}\to {\fl}.
\end{align*}
These maps are equivariant for the $G\times \check{G}$-action. By \cite[Proposition 2.22]{scholze_perfectoid_2013}, we know there is a basis $\ffrb_{\LT}$ for the analytic topology on $\calM_{\LT,\infty}$, such that for any $U\in\ffrb_{\LT}$:
\begin{enumerate}[(i)]
    \item $U$ is affinoid perfectoid.
    \item $U$ is the preimage of some affinoid subset $U_n\subset \calM_{\LT,n}$ for some $n$.
    \item For any $m\ge n$, let $U_m$ be the preimage of $U_n$. Then the natural map 
    \begin{align*}
        \dlim_{m\ge n}\calO_{\calM_{\LT,m}}(U_m)\to \calO_{\calM_{\LT,\infty}}(U)
    \end{align*}
    has dense image.
\end{enumerate}
Here we note as $U\in\ffrb_{\LT}$ is the preimage of some affinoid open subset of $\calM_{\LT,n}$, there exists a sufficiently small open compact subgroup of $G$ acting on $U$ by \cite[Lemma 2.2]{ScholzeLubinTateTower}.

Let $\calO_{\calM_{\LT,\infty}}$ be the completed structure sheaf of $\calM_{\LT,\infty}$. Define $\calO^{L\-\lan}_{\calM_{\LT,\infty}}\subset\calO_{\calM_{\LT,\infty}}$ to be the subsheaf consisting of locally $L$-analytic sections for the $G$-action. More precisely, for $U\in\ffrb_{\LT}$, there exists an open compact subgroup $G_U$ of $G$ acting on $U$. Then we put $\calO^{L\-\lan}_{\calM_{\LT,\infty}}(U):=\calO_{\calM_{\LT,\infty}}(U)^{(G_U,L)\-\lan}$. As taking locally $L$-analytic vectors is a left exact functor, we see $\calO^{L\-\lan}_{\calM_{\LT,\infty}}\subset\calO_{\calM_{\LT,\infty}}$ is indeed a subsheaf of $\calO_{\calM_{\LT,\infty}}$. Similarly we define $\calO_{\calM_{\LT,\infty}}^{\sm}$ as the subsheaf consisting of $G$-smooth sections. By \cite[Lemma 2.3.4]{camargo_geometric_2022}, the image of the natural map $\pi^{-1}\calO_{\fl}\to \calO_{\calM_{\LT,\infty}}$ lies in the subsheaf $\calO^{L\-\lan}_{\calM_{\LT,\infty}}$. Write $\pi=\pi_{\LT,\HT}$. There is a $G$-equivariant action of $\pi^{-1}\frg^0$ on $\calO^{L\-\lan}_{\calM_{\LT,\infty}}$ given by derivation. By the geometric Sen theory \cite[Theorem 4.3.3]{dospinescu2024jacquetlanglandsfunctorpadiclocally}, we know $\calO^{L\-\lan}_{\calM_{\LT,\infty}}$ is killed by $\pi^{-1}\frn^0$. Hence we get an action of $\pi^{-1}\frm^0$ on $\calO^{L\-\lan}_{\calM_{\LT,\infty}}$. Let $\calO^{L\-\lan,\frm^0=0}_{\calM_{\LT,\infty}}\subset\calO^{L\-\lan}_{\calM_{\LT,\infty}}$ be the subsheaf such that the action of $\pi^{-1}\frm^0$ is trivial.  As $\frm^0=\frp^0/\frn^0\subset \frg^0/\frn^0$ is an ideal and $G$ is connected, we see $\calO^{L\-\lan,\frm^0=0}_{\calM_{\LT,\infty}}$ is a $G$-equivariant subsheaf of $\calO^{L\-\lan}_{\calM_{\LT,\infty}}$. We concern on the local structure of $\calO^{L\-\lan,\frm^0=0}_{\calM_{\LT,\infty}}$. Roughly speaking, we will show $\calO^{L\-\lan,\frm^0=0}_{\calM_{\LT,\infty}}$ is a completed tensor product of $\pi^{-1}\calO_{\fl}$ with $\calO_{\calM_{\LT,\infty}}^{\sm}$ over $C$.

Let $z_i$ for $i=0,1,...,d$ be the coordinates on $\fl$. Let $U\in\ffrb_{\LT}$ such that $z_d\neq 0$ on $U$. As the image of $\pi:\calM_{\LT,\infty}\to \fl$ lies in the Drinfeld space $\Omega^d\subset\fl$, the assumption $z_d\neq 0$ on $U$ is automatically satisfied. For $i=0,1,...,d-1$, put $x_i=\frac{z_{i}}{z_d}\in (\pi^{-1}\calO_{\fl})(U)$.  As $\calO_{\calM_{\LT,\infty}}^{\sm}(U)\subset \calO_{\calM_{\LT,\infty}}(U)$ has dense image, there exists $x_{i,n}\in \calO_{\calM_{\LT,\infty}}^{\sm}(U)$ for $i=0,1,...,d-1$ such that $||x_i-x_{i,n}||\le p^{-n}$. For $i,j=0,1,...,d$, let $E_{ij}\in\frg $ be the matrix such that only the $(i,j)$-entry is $1$ and others are $0$. As $E_{d,i}z_i=z_d$ and $E_{d,i}z_d=0$, a direct computation shows $E_{d,i}.x_i=1$.

\begin{theorem}\label{thm:structureofOm0}
Let $U\in\ffrb_{\LT}$ such that $z_d\neq 0$ on $U$. Let $s\in \calO^{L\-\lan,\frm^0=0}_{\calM_{\LT,\infty}}(U)$. Then there exists a sufficiently large integer $n\ge 0$, such that
\begin{align*}
    s=\sum_{i_0=0}^\infty\sum_{i_1=0}^\infty\cdots\sum_{i_{d-1}=0}^\infty c_{i_0...i_{d-1}}(x_{0}-x_{0,n})^{i_0}\cdots(x_{d-1}-x_{d-1,n})^{i_{d-1}}
\end{align*}
where $c_{i_0...i_{d-1}}\in\calO^{\sm}_{\calM_{\LT,\infty}}(U)$ and $c_{i_0...i_{d-1}}p^{n(i_0+\cdots i_{d-1})}\to 0$. 
\begin{proof}
The idea is to expand $s$ along $E_{d,i}$ for $i=0,1,...,d-1$, following \cite[4.3.5]{pan_locally_2022}. Let $s\in \calO^{L\-\lan,\frm^0=0}_{\calM_{\LT,\infty}}(U)$. For each $i=0,1,...,d-1$, define 
\begin{align*}
    D_i(s):=\sum_{l=0}^\infty(-1)^l(x_i-x_{i,n})^l\frac{(E_{d,i})^l.s}{l!}
\end{align*}
which converges in $\calO_{\calM_{\LT,\infty}}^{L\-\lan}(U)$. Later we will show $D_i(s)$ acturally lies in the smaller subspace $\calO_{\calM_{\LT,\infty}}^{L\-\lan,\frm^0=0}(U)$.
Using $E_{d,i}(x_i-x_{i,n})=1$, a direct computation shows $E_{d,i}.D_i(s)=0$ for all $s\in \calO^{L\-\lan,\frm^0=0}_{\calM_{\LT,\infty}}(U)$, and 
\begin{align*}
    s=\sum_{l=0}^\infty a_{i,l}(x_{i}-x_{i,n})^l
\end{align*}
with $a_{i,l}=D_i(\frac{(E_{d,i})^l.s}{l!})$. Here we note the coefficients $a_{i,l}$ are determined by $s$. From the construction we know the coefficients $a_{i,l}$ are killed by $E_{d,i}$ for all $l$. For $j\neq i$, we furthure express the coefficients $a_{i,l}$ along $E_{d,j}$, namely 
\begin{align*}
    a_{i,l}=\sum_{k=0}^\infty b_{i,l,j,l'}(x_j-x_{j,n})^{l'}
\end{align*}
with $b_{i,l,j,l'}=D_{j}(\frac{(E_{d,j})^{l'}.a_{i,l}}{l'!})$. As $[E_{d,i},E_{d,j}]=0$ for $i,j=0,1,...,d-1$, we see $b_{i,l,j,l'}$ is killed by $E_{d,i}$ and $E_{d,j}$. After we expand $s$ along each derivation $E_{d,i}$ for $i=0,1,...,d-1$, we arrive at an expression 
\begin{align*}
    s=\sum_{i_0=0}^\infty\sum_{i_1=0}^\infty\cdots\sum_{i_{d-1}=0}^\infty c_{i_0...i_{d-1}}(x_{0}-x_{0,n})^{i_0}\cdots(x_{d-1}-x_{d-1,n})^{i_{d-1}}
\end{align*}
such that $E_{d,i}c_{i_0...i_{d-1}}=0$ for all $i=0,1,...,d-1$. As $\frm^0$ acts trivially on $\calO_{\fl}$, we know $\frm^0$ acts trivially on $x_i-x_{i,n}$ for any $i=0,...,d-1$. Therefore,
\begin{align*}
    0=\sum_{i_0,...,i_{d-1}}(\frm^0 c_{i_0...i_{d-1}})(x_{0}-x_{0,n})^{i_0}\cdots(x_{d-1}-x_{d-1,n})^{i_{d-1}}
\end{align*}
This implies $\frm^0 c_{i_0...i_{d-1}}=0$. Indeed, the coefficients $c_{i_0...i_{d-1}}$ are determined by $s$. Combined this with the fact that $\calO_{\calM_{\LT,\infty}}^{L\-\lan}$ is killed by $\pi^{-1}\frn^0$, we see $\frp^0$ kills $c_{i_0...i_{d-1}}$. Hence there is an induced action of $\bar\frn^0=\frg^0/\frp^0$ on $c_{i_0...i_{d-1}}$. On $U$, $\bar\frn^0$ is generated by $E_{d,i}$ for $i=0,1,...,d-1$. Indeed, on the open locus $V_o=\{[z_0:z_1:\cdots:z_{d-1}:1]\}\subset\bbP^d$, the matrix $Z=\left( \begin{matrix} I_{d\times d}&0\\ (z_0,z_1,...,z_{d-1})&1 \end{matrix} \right)$ is a lifting of the point $z=[z_0:z_1:...:z_{d-1}:1]$ as $oZ=z$ where $o=[0:0:\cdots:0:1]$. Then $Z^{-1}E_{d,i}Z=E_{d,i}$ for $i=0,1,...,d$ generates $\bar\frn^0$ on $V_o$. From the construction we know $c_{i_0...i_{d-1}}$ is killed by $E_{d,i}$ for all $i=0,1,...,d-1$, hence it is killed by $\bar \frn^0$. This shows that $\frg^0$ acts trvially on $c_{i_0...i_{d-1}}$, which implies $c_{i_0...i_{d-1}}\in\calO^{\sm}_{\calM_{\LT,\infty}}(U)$.
\end{proof}
\end{theorem}

\begin{corollary}\label{cor:OsmOfldense}
For any $U\in\ffrb_{\LT}$, the image of $\calO^{\sm}_{\calM_{\LT,\infty}}(U)\ox_C\pi^{-1}\calO_{\fl}(U)$ inside $\calO^{L\-\lan,\frm^0=0}_{\calM_{\LT,\infty}}(U)$ is dense.
\begin{proof}
This follows directly from Theorem \ref{thm:structureofOm0}.
\end{proof}
\end{corollary}

To simplify the notation, we denote $\calO_{\LT}$ by the sheaf $\calO^{L\-\lan,\frm^0=0}_{\calM_{\LT,\infty}}$ on $\calM_{\LT,\infty}$. Put $\calO_{\LT}^{\sm}:=\calO^{G\-\sm}_{\calM_{\LT,\infty}}$. Let $\Omega_{\calM_{\LT,n}}^k$ be the sheaf of $k$-differential forms on $\calM_{\LT,n}$ for $k=0,1,...,d$ and put $\Omega_{\LT}^{k,\sm}:=\dlim_n\pi_n^{-1}\Omega_{\calM_{\LT,n}}^k$ with $\pi_n:\calM_{\LT,\infty}\to \calM_{\LT,n}$. Clearly $\calO_{\LT}^{\sm}=\Omega_{\LT}^{0,\sm}$ and $\Omega_{\LT}^{k,\sm}=\wedge^k_{\calO_{\LT}^{\sm}}\Omega_{\LT}^{1,\sm}$.
\begin{proposition}\label{prop:dLT}
There exists a differential operator $d:\calO_{\LT}\to \calO_{\LT}\ox_{\calO_{\LT}^{\sm}}\Omega^{1,\sm}_{\LT}$, such that
\begin{enumerate}[(i)]
    \item $d$ is given by the usual derivation on $\calO^{\sm}_{\LT}$, 
    \item $d$ is $\pi^{-1}\calO_{\fl}$-linear.
\end{enumerate}
Moreover, $d$ is uniquely determined by these two properties up to constants.
\begin{proof}
Define $d|_{\calO_{\LT}^{\sm}}$ as the differential map on finite levels $\calO_{\LT}^{\sm}\to \Omega^{1,\sm}_{\LT}$, and define $d|_{\pi^{-1}\calO_{\fl}}$ to be the zero map. By Theorem \ref{thm:structureofOm0}, for any $U\in\ffrb_{\LT}$ such that $z_d\neq 0$ on $U$, we may write any section $s\in \calO_{\LT}(U)$ of the form for some sufficiently large $n$:
\begin{align*}
    s=\sum_{i_0=0}^\infty\sum_{i_1=0}^\infty\cdots\sum_{i_{d-1}=0}^\infty c_{i_0...i_{d-1}}(x_{0}-x_{0,n})^{i_0}\cdots(x_{d-1}-x_{d-1,n})^{i_{d-1}}
\end{align*}
where $c_{i_0...i_{d-1}}\in\calO^{\sm}_{\LT}(U)$ and $c_{i_0...i_{d-1}}p^{n(i_0+\cdots i_{d-1})}\to 0$. Then we define $d(s)$ using the Lebnitz rule:
\begin{align*}
    d(s)=\sum_{i_0=0}^\infty\sum_{i_1=0}^\infty\cdots\sum_{i_{d-1}=0}^\infty d(c_{i_0...i_{d-1}})(x_{0}-x_{0,n})^{i_0}\cdots(x_{d-1}-x_{d-1,n})^{i_{d-1}}\\-\sum_{j=0}^{d-1}(x_{0}-x_{0,n})^{i_0}\cdots (i_j(x_j-x_{j,n})^{i_j-1}dx_{j,n})\cdots(x_{d-1}-x_{d-1,n})^{i_{d-1}}.
\end{align*}
As $d:\calO_{\LT}^{\sm}(U)\to \Omega^{1,\sm}_{\LT}(U)$ is continuous, we see the expression defining $d(s)$ converges. Besides, since the image of $\calO^{\sm}_{\LT}(U)\ox_C\pi^{-1}\calO_{\fl}(U)$ is dense in $\calO_{\LT}(U)$ by Corollary \ref{cor:OsmOfldense}, $d(s)$ is independent of the choice of $n$ and $x_{j,n}$ for $j=0,1,...,d-1$. Also $d$ commutes with restriction of $U$ to its rational subsets. In general, one can use $G$-equivariancy to reduce to the case where $z_d\neq 0$ on $U$. This shows the existence and the uniqueness of $d$.
\end{proof}
\end{proposition}

\begin{corollary}\label{cor:dkLT}
For $k\ge 0$, there exists a differential operator $d^k:\calO_{\LT}\ox_{\calO_{\LT}^{\sm}}\Omega_{\LT}^{k,\sm}\to \calO_{\LT}\ox_{\calO_{\LT}^{\sm}}\Omega_{\LT}^{k+1,\sm}$, such that
\begin{enumerate}[(i)]
    \item the restriction of $d^k$ on $\Omega_{\LT}^{k,\sm}$ is given by the usual derivation $\Omega^{k,\sm}_{\LT}\to \Omega_{\LT}^{k+1,\sm}$, 
    \item $d^k$ is $\pi^{-1}\calO_{\fl}$-linear.
\end{enumerate}
Moreover, $d^k$ is uniquely determined by these two properties up to constants.
\begin{proof}
Let $U\in\ffrb_{\LT}$ such that $z_d\neq 0$ on $U$. For $s\in (\calO_{\LT}\ox_{\calO_{\LT}^{\sm}}\Omega_{\LT}^{k,\sm})(U)$, we may assume $s=\sum_i s_i\ox \omega_i$ for some $s_i\in \calO_{\LT}(U)$ and $\omega_i\in \Omega_{\LT}^{k,\sm}(U)$. Define 
\begin{align*}
    d^k(s)=\sum_i d(s_i)\wedge\omega_i+s_id\omega_i
\end{align*}
where $d(s_i)$ is defined as in Proposition \ref{prop:dLT}. Then one directly check $d^k$ is well-defined and satisfies the desired properties.
\end{proof}
\end{corollary}

\begin{corollary}
For $k\ge 0$, $d^{k+1}\comp d^k=0$ on $\calO_{\LT}\ox_{\calO_{\LT}^{\sm}}\Omega_{\LT}^{k,\sm}$.
\begin{proof}
As $d^k$ is determined by its restriction on the image of $\Omega_{\LT}^{k,\sm}\ox_C\pi^{-1}\calO_{\fl}$, this follows from $d^{k+1}\comp d^k=0$ on $\Omega_{\LT}^{k,\sm}$.
\end{proof}
\end{corollary}

We also define differential operators on $\calO_{\LT}$ along the flag variety. Let $\Omega^k_{\fl}$ be the sheaf of $k$-differential forms on $\fl$.
\begin{proposition}\label{prop:dbarLT}
For $k\ge 0$, there exists a differential operator $\bar d^k:\calO_{\LT}\ox_{\pi^{-1}\calO_{\fl}}\pi^{-1}\Omega_{\fl}^k\to \calO_{\LT}\ox_{\pi^{-1}\calO_{\fl}}\pi^{-1}\Omega_{\fl}^{k+1}$, such that
\begin{enumerate}[(i)]
    \item The restriction of $\bar d$ on $\pi^{-1}\Omega^k_{\fl}$ is induced by the usual derivation $\Omega^k_{\fl}\to \Omega_{\fl}^{k+1}$,
    \item $\bar d^k$ is $\calO^{\sm}_{\LT}$-linear.
\end{enumerate}
Moreover, $\bar d^k$ is uniquely determined by these two properties up to constants.
\begin{proof}
First we handle the case where $k=0$. As $\frp^0$ acts trivially on $\calO_{\LT}$, we get an action of $\bar\frn^0\isom \frg^0/\frp^0$ on $\calO_{\LT}$. Taking dual, we get 
\begin{align*}
    \bar d:\calO_{\LT}\to \calO_{\LT}\ox_{\pi^{-1}\calO_{\fl}}\pi^{-1}\Omega_{\fl}^1
\end{align*}
as $(\bar\frn^0)^{\vee}\isom \Omega_{\fl}^1$. From the construction we see that $\bar d$ is $\calO^{\sm}_{\LT}$-linear, and $\bar d$ is determined by these two properties by Corollary \ref{cor:OsmOfldense}. Explicitly, for $U\in\ffrb_{\LT}$ with $z_d\neq 0$ on $U$ and $s\in \calO_{\LT}(U)$, we have
\begin{align*}
    \bar d(s)=\sum_{i=0}^{d-1}(E_{d,i}.s)dx_i.
\end{align*}
The general case follows by similar treatments in Corollary \ref{cor:dkLT}.
\end{proof}
\end{proposition}

As $\bar d^k$ is determined by its restriction to $\calO^{\sm}_{\LT}\ox_C\pi^{-1}\Omega^k_{\fl}$, we see 
\begin{corollary}
For $k\ge 0$, $\bar d^{k+1}\comp \bar d^k=0$ on $\calO_{\LT}\ox_{\pi^{-1}\calO_{\fl}}\pi^{-1}\Omega_{\fl}^k$.
\end{corollary}

Since $d$ is $\pi^{-1}\calO_{\fl}$-linear, we may twist $d$ by $\pi^{-1}\Omega_{\fl}^k$. Similarly, as $\bar d$ is $\calO_{\LT}^{\sm}$-linear, we may twist $\bar d$ by $\Omega_{\LT}^{\sm}$. These constructions lead to the following commutative diagram. The commutativity of the diagram follows from the construction of these differential operators.

\begin{small}
\begin{tikzcd}
\calO_{\LT} \arrow[r] \arrow[d] & {\calO_{\LT}\ox_{\calO_{\LT}^{\sm}}\Omega_{\LT}^{1,\sm}} \arrow[r] \arrow[d] & \cdots \arrow[r] & {\calO_{\LT}\ox_{\calO_{\LT}^{\sm}}\Omega_{\LT}^{d,\sm}} \arrow[d] \\
\calO_{\LT}\ox_{\pi^{-1}\calO_{\fl}}\pi^{-1}\Omega_{\fl}^1 \arrow[d] \arrow[r]           & \Omega_{\LT}^{1,\sm}\ox_{\calO_{\LT}^{\sm}}\calO_{\LT}\ox_{\pi^{-1}\calO_{\fl}}\pi^{-1}\Omega_{\fl}^1 \arrow[r] \arrow[d]                                                        & \cdots \arrow[r] & \Omega_{\LT}^{d,\sm}\ox_{\calO_{\LT}^{\sm}}\calO_{\LT}\ox_{\pi^{-1}\calO_{\fl}}\pi^{-1}\Omega_{\fl}^1 \arrow[d]                                                        \\
\vdots \arrow[d]                & \vdots \arrow[d]                                                             &                  & \vdots \arrow[d]                                                   \\
\calO_{\LT}\ox_{\pi^{-1}\calO_{\fl}}\pi^{-1}\Omega_{\fl}^d \arrow[r]                     & \Omega_{\LT}^{1,\sm}\ox_{\calO_{\LT}^{\sm}}\calO_{\LT}\ox_{\pi^{-1}\calO_{\fl}}\pi^{-1}\Omega_{\fl}^d \arrow[r]                                                                  & \cdots \arrow[r] & \Omega_{\LT}^{d,\sm}\ox_{\calO_{\LT}^{\sm}}\calO_{\LT}\ox_{\pi^{-1}\calO_{\fl}}\pi^{-1}\Omega_{\fl}^d                                                                 
\end{tikzcd}
\end{small}
with horizontal arrows given by twists of $d^k$, and verticle arrows given by twists of $\bar d^k$. 

Now we want to calculate the cohomology of these differential operators. Since $\bar d$ on $\calO_{\LT}$ is given by the action of $\bar \frn^0$, we see $\ker \bar d$ is killed by $\frg^0$. Therefore, $\ker \bar d=\calO_{\LT}^{\sm}$. In fact these differential operators $\bar d$ will give a resolution of $\calO_{\LT}^{\sm}$.
\begin{theorem}\label{thm:dbarLTresolution}
The complex 
\begin{align*}
    0\to \calO^{\sm}_{\LT}\to \calO_{\LT}\to \calO_{\LT}\ox_{\pi^{-1}\calO_{\fl}}\pi^{-1}\Omega_{\fl}^1\to \cdots\to  \calO_{\LT}\ox_{\pi^{-1}\calO_{\fl}}\pi^{-1}\Omega_{\fl}^d\to 0
\end{align*}
is exact.
\begin{proof}
By Theorem \ref{thm:structureofOm0}, we see that on a sufficiently small $U\in\ffrb_{\LT}$, $\calO_{\LT}(U)$ is isomorphic to $\calO_{\LT}^{\sm}(U)\hat\ox_C\calC^{L\-\lan}(\bar\frn_U,C)$ as $\bar\frn_U$-modules, where $\bar\frn_U\subset\frg$ is a Lie subalgebra trivializing the vector bundle $\pi^{-1}\bar\frn^0|_U$. Here, the completed tensor product is defined as the completed tensor product of two LB spaces. Also the space of locally $L$-analytic functions on $\bar\frn_U$ is defined as in \cite[Definition 2.1.2]{jacinto2023solidlocallyanalyticrepresentations}. Under this identification, 
\begin{align*}
    0\to \calO^{\sm}_{\LT}(U)\to \calO_{\LT}(U)\to (\calO_{\LT}\ox_{\pi^{-1}\calO_{\fl}}\pi^{-1}\Omega_{\fl}^1)(U)\to \cdots\to  (\calO_{\LT}\ox_{\pi^{-1}\calO_{\fl}}\pi^{-1}\Omega_{\fl}^d)(U)\to 0
\end{align*}
is identified with the complex calculating the $\bar\frn_U$-cohomology of $\calO_{\LT}^{\sm}(U)\hat\ox_C\calC^{L\-\lan}(\bar\frn_U,C)$, see for example \cite[Proposition 3.1]{STduality}. Then by the locally analytic Poincar\'e lemma, we see the complex is exact.
\end{proof}
\end{theorem}

\subsection{The Drinfeld space at infinite level}
Now we turn to the Drinfeld side. Recall $D=D_{L,\frac{1}{d+1}}$ is a division algebra over $L$ of invariant $\frac{1}{d+1}$. Let $H_1$ be a special formal $\calO_{D}$-module over $\bar\bbF_q$ of $\calO_L$-height $(d+1)^2$, which is unique up to quasi-isogeny. Let $\ffrm_{\Dr,0}$ be the functor from $\Nilp_{\calO_{\breve{L}}}$ to the category of sets, which assigns an $\calO_{\breve{L}}$-algebra $R$ to the set of equivalence classes of pairs $(G,\rho)$, where $G$ is a $p$-divisible special formal $\calO_{D}$-module over $R$, and $\rho:H_1\ox_{\bar\bbF_q}R/\varpi\dasharrow G\ox_R R/\varpi$ is an $\calO_{D}$-equivariant quasi-isogeny. Note that the property of being special implies compatible, i.e., the action of $\calO_L\to R$ on the Lie algebra of $G$ coincides with the action of $\calO_L\subset\calO_D$ on the Lie algebra. Drinfeld \cite{drinfeld_coverings_1976} showed the functor $\ffrm_{\Dr,0}$ is represented by a formal scheme over $\Spf \calO_{\breve{L}}$. Let $\calM_{\Dr,0}$ denote the base change of the adic generic fiber from $\breve{L}$ to $C$. Let $(\calG',\rho')$ be the universal $\varpi$-divisible special formal $\calO_{D}$-module on $\calM_{\Dr,0}$. The $\varpi$-adic Tate module defines a $\bbZ_p$-local system equipped with an action of $D$ on the \'etale site of $\calM_{\Dr,0}$. For $n\ge 0$, we get an \'etale Galois covering $\calM_{\Dr,n}$ of $\calM_{\Dr,0}$ with Galois group $(\calO_{D}/\varpi^n\calO_D)^\times=\calO_{D}^\times/(1+\varpi^n\calO_{D})$ by considering $\calO_{D}$-equivariant trivializations of $V_{\Dr}/\varpi^n V_{\Dr}$.  By \cite{scholze_moduli_2013}, there exists a perfectoid space $\calM_{\Dr,\infty}$ over $C$ such that 
\[
    \calM_{\Dr,\infty}\sim\ilim_n\calM_{\Dr,n}
\]
There is a natural continuous right action of $\calO_{D}^\times$ on $\calM_{\Dr,\infty}$, which can be extended naturally to an action of $\check{G}=D^\times$. Also ${G}$ acts on each layer $\calM_{\Dr,n}$, $n\ge 0$. According to the height of the quasi-isogeny in the moduli problem, the space $\calM_{\Dr,\infty}$ decomposes $\calO_D^\times$-equivariantly as a disjoint union $\calM_{\Dr,\infty}\isom \sqcup_{i\in\bbZ}\calM_{\Dr,\infty}^{(i)}$. The Hodge filtration of the Dieudonn\'e module of the universal $p$-divisible group $(\calG',\rho')$ defines the Gross-Hopkins period map \cite{hopkins_equivariant_1994}
\begin{align*}
    \pi_{\Dr,\GM}:\calM_{\Dr,\infty}\to {\fl},
\end{align*}
and the Hodge filtration of the rational Tate module of the universal $p$-divisible group $(\calG',\rho')$ defines the Hodge-Tate period map \cite{scholze_moduli_2013}
\begin{align*}
    \pi_{\Dr,\HT}:\calM_{\Dr,\infty}\to \check{\fl}.
\end{align*}
These maps are equivariant for the $\check{G}\times G$-action.

By \cite[Proposition 7.2.2, Theorem 7.2.3]{scholze_moduli_2013}, we have the following duality theorem.
\begin{theorem}[Scholze-Weinstein]\label{thm:SWduality}
There is a $G\times \check G$-equivariant isomorphism of perfectoid spaces over $C$:
\begin{align*}
    \calM_{\LT,\infty}\isom \calM_{\Dr,\infty}
\end{align*}
such that $\pi_{\LT,\HT}$ is identified with $\pi_{\Dr,\GM}$, and $\pi_{\LT,\GM}$ is identified with $\pi_{\Dr,\HT}$.
\end{theorem}
In particular, the completed structure sheaves $\calO_{\calM_{\Dr,\infty}}$, $\calO_{\calM_{\LT,\infty}}$ are $G\times \check G$-equivariantly isomorphic to each other. Similar to the treatment on the Lubin-Tate side, we will use the period map $\pi_{\Dr,\HT}:\calM_{\Dr,\infty}\to \check{\fl}$ to define a subsheaf of $\calO_{\calM_{\Dr,\infty}}$, and study some basic properties. By \cite[Proposition 2.22]{scholze_perfectoid_2013}, there is a basis $\ffrb_{\Dr}$ for the analytic topology on $\calM_{\Dr,\infty}$, such that for any $U\in\ffrb_{\Dr}$:
\begin{enumerate}[(i)]
    \item $U$ is affinoid perfectoid.
    \item $U$ is the preimage of some affinoid subset $U_n\subset \calM_{\Dr,n}$ for some $n$.
    \item For any $m\ge n$, let $U_m$ be the preimage of $U_n$. Then the natural map 
    \begin{align*}
        \dlim_{m\ge n}\calO_{\calM_{\Dr,m}}(U_m)\to \calO_{\calM_{\Dr,\infty}}(U)
    \end{align*}
    has dense image.
\end{enumerate}
Here we note as $U\in\ffrb_{\Dr}$ is the preimage of some affinoid open subset of $\calM_{\Dr,n}$, there exists a sufficiently small open compact subgroup of $\check G$ acting on $U$.

Let $\calO_{\calM_{\Dr,\infty}}$ be the completed structure sheaf of $\calM_{\Dr,\infty}$. Define $\calO^{L\-\lan}_{\calM_{\Dr,\infty}}\subset\calO_{\calM_{\Dr,\infty}}$ to be the subsheaf consisting of locally $L$-analytic sections for the $\check G$-action. More precisely, for $U\in\ffrb_{\Dr}$, there exists an open compact subgroup $\check G_U$ of $\check G$ acting on $U$. Then we put $\calO^{L\-\lan}_{\calM_{\Dr,\infty}}(U):=\calO_{\calM_{\Dr,\infty}}(U)^{(\check G_U,L)\-\lan}$. As taking locally $L$-analytic vectors is a left exact functor, we see $\calO^{L\-\lan}_{\calM_{\Dr,\infty}}\subset\calO_{\calM_{\Dr,\infty}}$ is indeed a subsheaf of $\calO_{\calM_{\Dr,\infty}}$. Similarly we define $\calO_{\calM_{\Dr,\infty}}^{\sm}$ as the subsheaf consisting of $\check G$-smooth sections. By \cite[Lemma 2.3.4]{camargo_geometric_2022}, the image of the natural map $\check \pi^{-1}\calO_{\check \fl}\to \calO_{\calM_{\Dr,\infty}}$ lies in the subsheaf $\calO^{L\-\lan}_{\calM_{\Dr,\infty}}$. Write $\check \pi=\pi_{\Dr,\HT}$. There is a $\check G$-equivariant action of $\check \pi^{-1}\check\frg^0$ on $\calO^{L\-\lan}_{\calM_{\Dr,\infty}}$ given by derivation. By the geometric Sen theory \cite[Theorem 4.3.3]{dospinescu2024jacquetlanglandsfunctorpadiclocally}, we know $\calO^{L\-\lan}_{\calM_{\Dr,\infty}}$ is killed by $\check\pi^{-1}\check\frn^0$. Hence we get an action of $\check\pi^{-1}\check\frm^0$ on $\calO^{L\-\lan}_{\calM_{\LT,\infty}}$. Let $\calO^{L\-\lan,\frm^0=0}_{\calM_{\Dr,\infty}}\subset\calO^{L\-\lan}_{\calM_{\Dr,\infty}}$ be the subsheaf such that the action of $\check\pi^{-1}\check\frm^0$ is trivial. Similar to the Lubin-Tate case, we will show $\calO^{L\-\lan,\check\frm^0=0}_{\calM_{\Dr,\infty}}$ is (roughly speaking) a completed tensor product of $\pi^{-1}\calO_{\check\fl}$ with $\calO_{\calM_{\Dr,\infty}}^{\sm}$ over $C$.

Let $w_i$ for $i=0,1,...,d$ be the coordinates on $\check\fl$. Let $U\in\ffrb_{\Dr}$ such that $w_d\neq 0$ on $U$. For $i=0,1,...,d-1$, put $y_i=\frac{w_{i}}{w_d}\in (\check\pi^{-1}\calO_{\check\fl})(U)$.  As $\calO_{\calM_{\Dr,\infty}}^{\sm}(U)\subset \calO_{\calM_{\Dr,\infty}}(U)$ has dense image, there exists $y_{i,n}\in \calO_{\calM_{\Dr,\infty}}^{\sm}(U)$ for $i=0,1,...,d-1$ such that $||y_i-y_{i,n}||\le p^{-n}$. For $i,j=0,1,...,d$, let $E_{ij}\in\check\frg$ be the matrix such that only the $(i,j)$-entry is $1$ and others are $0$. As $E_{d,i}w_i=w_d$ and $E_{d,i}w_d=0$, a direct computation shows $E_{d,i}.y_i=1$.

\begin{theorem}\label{thm:structureofOm0check}
Let $U\in\ffrb_{\Dr}$ such that $w_d\neq 0$ on $U$. Let $s\in \calO^{L\-\lan,\check\frm^0=0}_{\calM_{\Dr,\infty}}(U)$. Then there exists a sufficiently large integer $n\ge 0$, such that
\begin{align*}
    s=\sum_{i_0=0}^\infty\sum_{i_1=0}^\infty\cdots\sum_{i_{d-1}=0}^\infty c_{i_0...i_{d-1}}(y_{0}-y_{0,n})^{i_0}\cdots(y_{d-1}-y_{d-1,n})^{i_{d-1}}
\end{align*}
where $c_{i_0...i_{d-1}}\in\calO^{\sm}_{\calM_{\Dr,\infty}}(U)$ and $c_{i_0...i_{d-1}}p^{n(i_0+\cdots i_{d-1})}\to 0$. 
\begin{proof}
The proof follows the same approach as Theorem \ref{thm:structureofOm0}.
\end{proof}
\end{theorem}

\begin{corollary}\label{cor:OsmOfldensecheck}
For any $U\in\ffrb_{\Dr}$, the image of $\calO^{\sm}_{\calM_{\Dr,\infty}}(U)\ox_C\check\pi^{-1}\calO_{\check\fl}(U)$ is dense in $\calO^{L\-\lan,\check\frm^0=0}_{\calM_{\Dr,\infty}}(U)$.
\begin{proof}
Using the $\check G$-equivariance we may assume $w_d\neq 0$ on $U$. Then this follows directly from Theorem \ref{thm:structureofOm0check}.
\end{proof}
\end{corollary}

Recall the $G\times \check G$-equivariant isomorphism induces an isomorphism $\calO_{\calM_{\LT,\infty}}\isom \calO_{\calM_{\Dr,\infty}}$.
\begin{theorem}\label{thm:isom}
We have a $G\times \check G$-equivariant isomorphism $\calO_{\calM_{\LT,\infty}}^{L\-\lan,\frm^0=0}\isom \calO_{\calM_{\Dr,\infty}}^{L\-\lan,\check\frm^0=0}$.
\begin{proof}
By changing the role of Lubin-Tate and Drinfeld we see that it suffices to show $\calO_{\calM_{\LT,\infty}}^{L\-\lan,\frm^0=0}(U)\subset \calO_{\calM_{\Dr,\infty}}^{L\-\lan,\check\frm^0=0}(U)$ for any $U\in\ffrb_{\LT}$. Before we do the checking, we note that there is an action of a sufficiently small open compact subgroup $\check{K}_U$ of $\check{G}$ on $U$, as $U$ is quasi-compact so that it can be covered by finitely many open subsets in $\ffrb_{\Dr}$. As $\check G$ acts trivially on $\fl$, we know $\check{G}$ acts trivially on $x_i$ for $i=0,1,...,d-1$. Besides, the $\check{K}_U$-action on $\calO_{\calM_{\LT,\infty}}^{\sm}(U)$ is locally $L$-analytic. Indeed, we may cover $U$ by finitely many open subsets $U_i$ in $\ffrb_{\Dr}$, and the $\check{K}_U$-action on $\calO_{\calM_{\LT,\infty}}^{\sm}(U_i)$ is locally $L$-analytic by \cite[Lemma 2.3.4]{camargo_geometric_2022}. Therefore, using similar arguments as in \cite[Corollary 5.2.10]{pan_locally_2022-1}, we know $\calO_{\calM_{\LT,\infty}}^{L\-\lan,\frm^0=0}(U)\subset \calO_{\calM_{\Dr,\infty}}^{L\-\lan}(U)$. More precisely, let 
\begin{align*}
    s=\sum_{i_0=0}^\infty\sum_{i_1=0}^\infty\cdots\sum_{i_{d-1}=0}^\infty c_{i_0...i_{d-1}}(x_{0}-x_{0,n})^{i_0}\cdots(x_{d-1}-x_{d-1,n})^{i_{d-1}}
\end{align*}
be a section in $\calO_{\calM_{\LT,\infty}}^{L\-\lan,\frm^0=0}(U)$. After shrinking $\check K_U$ if necessary, we may assume the $\check{K}_U$-analytic norm on $c_{i_0\cdots i_{d-1}}$ and $x_i-x_{i,n}$ for $i=0,1,...,d-1$ is the same as the usual norm of them inside $\calO_{\calM_{\LT,\infty}}(U)$ by \cite[Lemma 2.1.5 (3)]{pan_locally_2022}. Then such a section $s$ should be $\check K_U$-analytic. It remains to show $\calO_{\calM_{\LT,\infty}}^{L\-\lan,\frm^0=0}(U)$ is killed by $\check{\frm}^0$-action. Recall the natural projection maps $\calM_{\LT,n}\to \calM_{\LT,0}$ and $\pi_{\LT,\GM,0}:\calM_{\LT,0}\to \check\fl$ are \'etale maps. Then we see $\check \frm^0$ acts trivially on $\calO_{\calM_{\LT,\infty}}^{\sm}$ since $\check \frm^0$ acts trivially on $\calO_{\check\fl}$. Since $\check G$ acts trivially on $x_i$ for $i=0,1,...,d-1$, we deduce $\check \frm^0$ acts trivially on $x_i$. Then from the local expansion Theorem \ref{thm:structureofOm0} we deduce $\calO_{\calM_{\LT,\infty}}^{L\-\lan,\frm^0=0}(U)\subset \calO_{\calM_{\Dr,\infty}}^{L\-\lan,\check\frm^0=0}(U)$.
\end{proof}
\end{theorem}

\begin{corollary}\label{cor:OOcheckdense}
Let $U\in\ffrb_{\LT}$ be an open subset. Then the image of $\calO_{\calM_{\LT,\infty}}^{\sm}(U)\ox_C\calO_{\calM_{\Dr,\infty}}^{\sm}(U)$ in $\calO_{\calM_{\LT,\infty}}^{L\-\lan,\frm^0=0}(U)$ is dense.
\begin{proof}
By Theorem \ref{thm:isom}, we know the image of $\calO_{\calM_{\LT,\infty}}^{\sm}(U)\ox_C\calO_{\calM_{\Dr,\infty}}^{\sm}(U)$ lies in $\calO_{\calM_{\LT,\infty}}^{L\-\lan,\frm^0=0}(U)$. Moreover, $\pi^{-1}\calO_{\fl}(U)\to \calO_{\calM_{\Dr,\infty}}^{\sm}(U)$ is injective as $\calM_{\Dr,0}$ is $\bbZ$-copy of the Drinfeld space $\Omega^d\subset \fl$. Then we can deduce the density result from Corollary \ref{cor:OsmOfldense}.
\end{proof}
\end{corollary}

For simplicity, we denote $\calO_{\Dr}$ by the sheaf $\calO^{L\-\lan,\check\frm^0=0}_{\calM_{\Dr,\infty}}$ on $\calM_{\Dr,\infty}$. Put $\calO_{\Dr}^{\sm}:=\calO^{\check G\-\sm}_{\calM_{\Dr,\infty}}$. Let $\Omega_{\calM_{\Dr,n}}^k$ be the sheaf of $k$-differential forms on $\calM_{\Dr,n}$ for $k=0,1,...,d$ and put $\Omega_{\Dr}^{k,\sm}:=\dlim_n\pi_n^{-1}\Omega_{\calM_{\Dr,n}}^k$ with $\pi_n:\calM_{\Dr,\infty}\to \calM_{\Dr,n}$. Clearly $\calO_{\Dr}^{\sm}=\Omega_{\Dr}^{0,\sm}$ and $\Omega_{\Dr}^{k,\sm}=\wedge^k_{\calO_{\Dr}^{\sm}}\Omega_{\Dr}^{1,\sm}$. Using similar methods in Corollary \ref{cor:dkLT}, Proposition \ref{prop:dbarLT}, we deduce the following two propositions.

\begin{proposition}
For $k\ge 0$, there exists a differential operator $\check d^k:\calO_{\Dr}\ox_{\calO_{\Dr}^{\sm}}\Omega_{\Dr}^{k,\sm}\to \calO_{\Dr}\ox_{\calO_{\Dr}^{\sm}}\Omega_{\Dr}^{k+1,\sm}$, such that
\begin{enumerate}[(i)]
    \item the restriction of $\check d^k$ on $\Omega_{\Dr}^{k,\sm}$ is given by the usual derivation $\Omega^{k,\sm}_{\Dr}\to \Omega_{\Dr}^{k+1,\sm}$, 
    \item $\check d^k$ is $\pi^{-1}\calO_{\check \fl}$-linear.
\end{enumerate}
Moreover, $\check d^k$ is uniquely determined by these two properties up to constants, and for $k\ge 0$, $\check d^{k+1}\comp \check d^k=0$.
\end{proposition}

\begin{proposition}\label{prop:dDrbar}
For $k\ge 0$, there exists a differential operator $\check {\bar  d}^k:\calO_{\Dr}\ox_{\check\pi^{-1}\calO_{\check\fl}}\check\pi^{-1}\Omega_{\check\fl}^k\to \calO_{\Dr}\ox_{\check\pi^{-1}\calO_{\check\fl}}\check\pi^{-1}\Omega_{\check\fl}^{k+1}$, such that
\begin{enumerate}[(i)]
    \item The restriction of $\check {\bar  d}$ on $\check\pi^{-1}\Omega^k_{\check\fl}$ is induced by the usual derivation $\Omega^k_{\check\fl}\to \Omega_{\check\fl}^{k+1}$,
    \item $\check {\bar  d}^k$ is $\calO^{\sm}_{\Dr}$-linear.
\end{enumerate}
Moreover, $\check {\bar  d}^k$ is uniquely determined by these two properties up to constants, and for $k\ge 0$, we have $\check {\bar  d}^{k+1}\comp \check {\bar  d}^k=0$.
\end{proposition}

Finally, similar to Theorem \ref{thm:dbarLTresolution}, we also obtain the following result.
\begin{theorem}\label{thm:dbarDrresolution}
The complex 
\begin{align*}
    0\to \calO^{\sm}_{\Dr}\to \calO_{\Dr}\to \calO_{\Dr}\ox_{\check\pi^{-1}\calO_{\check\fl}}\check\pi^{-1}\check\Omega_{\fl}^1\to \cdots\to  \calO_{\Dr}\ox_{\check\pi^{-1}\calO_{\check\fl}}\check\pi^{-1}\Omega_{\check\fl}^d\to 0
\end{align*}
is exact.
\end{theorem}

The next theorem shows that swapping the roles of Lubin-Tate and Drinfeld leads to an identification of the corresponding differential operators. For $0\le k,l\le d$ integers, we have defined two differential operators
\begin{align*}
    d: \Omega_{\LT}^{k,\sm}\ox_{\calO_{\LT}^{\sm}}\calO_{\LT}\ox_{\pi^{-1}\calO_{\fl}}\pi^{-1}\Omega_{\fl}^{l}\to \Omega_{\LT}^{k+1,\sm}\ox_{\calO_{\LT}^{\sm}}\calO_{\LT}\ox_{\pi^{-1}\calO_{\fl}}\pi^{-1}\Omega_{\fl}^{l}
\end{align*}
and 
\begin{align*}
    \check{\bar d}: \check{\pi}^{-1}\Omega_{\check\fl}^{k}\ox_{\check{\pi}^{-1}\calO_{\check\fl}}\calO_{\Dr}\ox_{\calO_{\Dr}^{\sm}}\Omega_{\Dr}^{l,\sm}\to \check{\pi}^{-1}\Omega_{\check\fl}^{k+1}\ox_{\check{\pi}^{-1}\calO_{\check\fl}}\calO_{\Dr}\ox_{\calO_{\Dr}^{\sm}}\Omega_{\Dr}^{l,\sm}.
\end{align*}
Since $\pi_{\Dr,\GM,n}:\calM_{\Dr,n}\to \fl$ is an \'etale morphism, we see the natural map $\pi_{\Dr,\GM,n}^*\Omega^1_{\fl}\to \Omega^1_{\calM_{\Dr,n}}$ is an isomorphism. This shows 
\begin{align*}
    \Omega_{\LT}^{k,\sm}\ox_{\calO_{\LT}^{\sm}}\calO_{\LT}\ox_{\calO_{\Dr}^{\sm}}\Omega_{\Dr}^{l,\sm}&\isom \dlim_n \Omega_{\LT}^{k,\sm}\ox_{\calO_{\LT}^{\sm}}\calO_{\LT}\ox_{\pi_n^{-1}\calO_{\calM_{\Dr,n}}}\pi_n^{-1}\Omega^l_{\calM_{\Dr,n}}\\
    &\isom \dlim_n \Omega_{\LT}^{k,\sm}\ox_{\calO_{\LT}^{\sm}}\calO_{\LT}\ox_{\pi_n^{-1}\calO_{\calM_{\Dr,n}}}\pi_n^{-1}(\pi_{\Dr,\GM,n}^{-1}\Omega^l_{\fl}\ox_{\pi_{\Dr,\GM,n}^{-1}\calO_{\fl}} \calO_{\calM_{\Dr,n}})\\
    &\isom \Omega_{\LT}^{k,\sm}\ox_{\calO_{\LT}^{\sm}}\calO_{\LT}\ox_{\pi^{-1}\calO_{\fl}}\pi^{-1}\Omega_{\fl}^{l}.
\end{align*}
where $\pi_n:\calM_{\Dr,\infty}\to \calM_{\Dr,n}$ is the projection. Similarly, as $\pi_{\LT,\GM,n}:\calM_{\LT,n}\to \check\fl$ is an \'etale morphism, and $\pi_{\LT,\GM,0}:\calM_{\LT,0}\to \check\fl$ has local sections, we see the natural map $\pi_{\LT,\GM,n}^*\Omega^1_{\check \fl}\to \Omega^1_{\calM_{\LT,n}}$ is an isomorphism. Then we deduce 
\begin{align*}
    \check{\pi}^{-1}\Omega_{\check\fl}^{k}\ox_{\check{\pi}^{-1}\calO_{\check\fl}}\calO_{\Dr}\ox_{\calO_{\Dr}^{\sm}}\Omega_{\Dr}^{l,\sm}\isom \Omega_{\LT}^{k,\sm}\ox_{\calO_{\LT}^{\sm}}\calO_{\Dr}\ox_{\calO_{\Dr}^{\sm}}\Omega_{\Dr}^{l,\sm}.
\end{align*}

\begin{theorem}\label{thm:ddbar}
Under the above identifications and $\calO_{\LT}\isom \calO_{\Dr}$, the differential operators $d$ and $\check{\bar d}$ are identified up to non-zero constants.
\begin{proof}
It suffices to handle the case $k=0$ and $l=0$ as the general case follows by taking wedge powers. By Corollary \ref{cor:OsmOfldense} and Proposition \ref{prop:dLT}, it suffices to check that $\check{\bar d}$ induces the usual derivation on $\calO_{\LT}^{\sm}$, and is trivial on $\pi^{-1}\calO_{\fl}$. By Proposition \ref{prop:dDrbar}, we know $\check{\bar d}$ induces the usual derivation on $\check\pi^{-1}\calO_{\check{\fl}}$, and $\check{\bar d}$ is $\calO_{\Dr}^{\sm}$-linear. As $\pi_{\LT,\GM,n}:\calM_{\LT,n}\to \check{\fl}$ and $\pi_{\Dr,\GM,n}:\calM_{\Dr,n}\to \fl$ are \'etale, we see $\check{\bar{d}}$ satisfies the desired property. Hence $d$ equals to $\check{\bar d}$ up to scalars.
\end{proof}
\end{theorem}

\begin{corollary}\label{cor:dLTresolution}
For any $0\le k\le d$, there is an exact sequence 
\begin{align*}
0\to {\Omega_{\Dr}^{k,\sm}} \to  \calO_{\LT}\ox_{\calO_{\Dr}^{\sm}}\Omega_{\Dr}^{k,\sm} \to  \Omega_{\LT}^{1,\sm}\ox_{\calO_{\LT}^{\sm}}\calO_{\LT}\ox_{\calO_{\Dr}^{\sm}}\Omega_{\Dr}^{k,\sm} \to \cdots \to  \Omega_{\LT}^{d,\sm}\ox_{\calO_{\LT}^{\sm}}\calO_{\LT}\ox_{\calO_{\Dr}^{\sm}}\Omega_{\Dr}^{k,\sm}        \to 0                   .          
\end{align*}
induced by twists of $d$.
\begin{proof}
This follows from Theorem \ref{thm:ddbar} and Theorem \ref{thm:dbarDrresolution}.
\end{proof}
\end{corollary}

Hence we arrived at a diagram 
$$
\begin{small}
\begin{tikzcd}
 \calO_{\LT} \arrow[r] \arrow[d]       & {\calO_{\LT}\ox_{\calO_{\LT}^{\sm}}\Omega_{\LT}^{1,\sm}} \arrow[r] \arrow[d] & \cdots \arrow[r] & {\calO_{\LT}\ox_{\calO_{\LT}^{\sm}}\Omega_{\LT}^{d,\sm}} \arrow[d] \\
 \calO_{\LT}\ox_{\calO_{\Dr}^{\sm}}\Omega_{\Dr}^{1,\sm} \arrow[d] \arrow[r]                 & \Omega_{\LT}^{1,\sm}\ox_{\calO_{\LT}^{\sm}}\calO_{\LT}\ox_{\calO_{\Dr}^{\sm}}\Omega_{\Dr}^{1,\sm} \arrow[r] \arrow[d]                                                        & \cdots \arrow[r] & \Omega_{\LT}^{d,\sm}\ox_{\calO_{\LT}^{\sm}}\calO_{\LT}\ox_{\calO_{\Dr}^{\sm}}\Omega_{\Dr}^{1,\sm} \arrow[d]                                                        \\
 \vdots \arrow[d]                      & \vdots \arrow[d]                                                             &                  & \vdots \arrow[d]                                                   \\
 \calO_{\LT}\ox_{\calO_{\Dr}^{\sm}}\Omega_{\Dr}^{d,\sm} \arrow[r]                           & \Omega_{\LT}^{1,\sm}\ox_{\calO_{\LT}^{\sm}}\calO_{\LT}\ox_{\calO_{\Dr}^{\sm}}\Omega_{\Dr}^{d,\sm} \arrow[r]                                                                  & \cdots \arrow[r] & \Omega_{\LT}^{d,\sm}\ox_{\calO_{\LT}^{\sm}}\calO_{\LT}\ox_{\calO_{\Dr}^{\sm}}\Omega_{\Dr}^{d,\sm}                                                                 
\end{tikzcd}
\end{small}
$$
with horizontal rows are twists of $d$, and verticle rows are twists of $\bar d$. By modifying the sign of the differential operators, we can make it a double complex, and we denote it by $M^{\bullet,\bullet}$. Let $\DR_{\LT}$ be the de Rham complex $[\calO_{\LT}^{\sm}\to \Omega_{\LT}^{1,\sm}\to\cdots\to \Omega_{\LT}^{d,\sm}]$ and let $\DR_{\Dr}$ be the de Rham complex $[\calO_{\Dr}^{\sm}\to \Omega_{\Dr}^{1,\sm}\to\cdots\to \Omega_{\Dr}^{d,\sm}]$. It turns out that $M^{\bullet,\bullet}$ interpolates $\DR_{\LT}$ and $\DR_{\Dr}$.
\begin{theorem}
There is a quasi-isomorphism 
\begin{align*}
    \DR_{\LT}\aisom \Tot(M^{\bullet,\bullet})\ov{\sim}\ot\DR_{\Dr}
\end{align*}
of complexes of $G\times\check{G}$-equivariant sheaves on $\calM_{\LT,\infty}\isom\calM_{\Dr,\infty}$. 
\begin{proof}
By Corollary \ref{cor:dLTresolution} and Theorem \ref{thm:dbarLTresolution}, we see the horizontal rows of $M^{\bullet,\bullet}$ form a resolution of $\DR_{\Dr}$, and the vertical rows of $M^{\bullet,\bullet}$ form a resolution of $\DR_{\LT}$. Hence we deduce the desired isomorphisms by \cite[\href{https://stacks.math.columbia.edu/tag/0133}{Tag 0133}]{stacks-project}.
\end{proof}
\end{theorem}

\begin{proof}[Proof of \ref{thm:main}]
We follow the argument in \cite[Theorem 5.2.2]{dospinescu2024jacquetlanglandsfunctorpadiclocally} on handling compactly supported cohomology of these de Rham complexes. Let $\calM^B$ be the Berkovich space associated to $\calM:=\calM_{\LT,\infty}\isom\calM_{\Dr,\infty}$, and let $\DR_{\LT}^B$, $\DR_{\Dr}^B$ be the de Rham complexes on $\calM^B$ such that if $f:\calM\to \calM^B$ is the map from adic space to the associated Berkovich space \cite[Definition 13.7]{ScholzeDiamond} then $f^*\DR_{\LT}^B\isom \DR_{\LT}$ and $f^*\DR_{\Dr}^B\isom \DR_{\Dr}$. By \cite[Proposition 13.13]{ScholzeDiamond}, we know the isomorphism $\DR_{\LT}\isom \DR_{\Dr}$ induces an isomorphism $\DR_{\LT}^B\isom \DR_{\Dr}^B$. Now we use the six functor formalism on condensed anima developed in \cite{heyer20246functorformalismssmoothrepresentations}. In the following, we will let $f^{-1}$ denote the pullback functors in this formalism. Let $*$ denote the one point space, and we claim the natural map $\calM^B\to *$ is $\bbQ_p$-fine in the sense of \cite[Definition 3.5.17]{heyer20246functorformalismssmoothrepresentations}.Indeed, let $\calM_{\LT,n}^B$ be the Berkovich space associated to $\calM_{\LT,n}$, then $\calM_{\LT,n}^B$ is a locally finite dimensional Hausdorff space as it is the Berkovich space associated to a rigid analytic variety, and \cite[Theorem 4.8.9 (i)]{heyer20246functorformalismssmoothrepresentations} implies $\calM_{\LT,n}^B\to *$ is $\bbQ_p$-fine. We claim that $\calM_{\LT,\infty}^B\to \calM_{\LT,n}^B$ is a $G_n$-torsor. Indeed, we have $\calM_{\LT,\infty}^B/G_n=(\calM_{\LT,\infty}/G_n)^B=\calM_{\LT,n}^B$. As $G_n$ acts freely on $\calM_{\LT,\infty}^B$, we deduce $\calM_{\LT,\infty}^B\times_{\calM_{\LT,n}^B}\calM_{\LT,\infty}^B\isom \calM_{\LT,\infty}^B\times G_n$. Now since $\calM_{\LT,\infty}^B\to \calM_{\LT,n}^B$ is represented in profinite sets, we see $\calM_{\LT,\infty}^B\to \calM_{\LT,n}^B$ is also $\bbQ_p$-fine by \cite[3.4.11(ii)]{heyer20246functorformalismssmoothrepresentations}. Therefore, there is a well-defined lower shriek functor $f_!$ associated to the natural map $f:\calM^B\to *$ and 
\begin{align*}
    f_!\DR_{\LT}^B\isom f_!\DR_{\Dr}^B.
\end{align*}
Now we identify both sides with the desired de Rham cohomology groups. By definition, $\calO_{\calM_{\LT,\infty}^B}^{\sm}\isom \dlim_n\pi_n^{-1}\calO_{\calM_{\LT,n}^B}$ with $\pi_n:\calM_{\LT,\infty}^B\to \calM_{\LT,n}^B$ being the projection map. Similarly, there is a de Rham complex $\DR_{\LT,n}^B$ associated to $\calM_{\LT,n}^B$ and $\DR_{\LT}^B=\dlim_n\pi_n^{-1}\DR_{\LT,n}^B$. We have a cartesian diagram 
$$
\begin{tikzcd}
{\calM^B_{\LT,\infty}} \arrow[r, "\pi_n"] \arrow[d, "f"'] & {\calM_{\LT,n}^B} \arrow[d, "f_n"] \\
* \arrow[r, "q_n"]                                     & */G_n                             
\end{tikzcd}$$
where $f_n$ is induced by the $G_n$-torsor $\calM_{\LT,\infty}^B$ over $\calM_{\LT,n}^B$, with $G_n=\ker(\bbG(\calO_L)\to \bbG(\calO_L/\varpi^n\calO_L))$ and $q_n:*\to */G_n$ is $\bbQ_p$-proper by \cite[Proposition 5.3.2]{heyer20246functorformalismssmoothrepresentations}. By the prim base change \cite[Lemma 4.5.13]{heyer20246functorformalismssmoothrepresentations} we see $f_!\pi_n^{-1}\DR_{\LT,n}^B\isom q_n^{-1}f_{n,!}\DR_{\LT,n}^B$, and
\begin{align*}
    f_!\DR_{\LT}^B\isom \dlim_n f_!\pi_n^{-1}\DR_{\LT,n}^B\isom \dlim_n q_n^{-1}f_{n,!}\DR_{\LT,n}^B.
\end{align*}
Hence it remains to identify the right hand side with the compactly supported de Rham cohomology groups. For $m\ge n$ integers, consider the following commutative diagram
$$
\begin{tikzcd}
{\calM_{\LT,\infty}^B} \arrow[r, "\pi_m"] \arrow[d, "f"'] & {\calM_{\LT,m}^B} \arrow[r, "{\pi_{n,m}}"] \arrow[d, "f_{m}"] & {\calM_{\LT,n}^B} \arrow[d, "f_n"] \\
* \arrow[r, "q_m"] \arrow[rrd, "i"']                    & */G_{m} \arrow[r, "{q_{n,m}}"] \arrow[rd, "i_m"]            & */G_n \arrow[d, "i_n"]           \\
                                                        &                                                             & *                               
\end{tikzcd}
$$
with each square being cartesian. Then 
\begin{align*}
    q_n^{-1}f_{n,!}\DR_{\LT,n}^B\isom i_{n,!}q_{n,!}q_n^{-1}f_{n,!}\DR_{\LT,n}^B&\ov{(*)}{=}\dlim_{m\ge n}i_{n,!}q_{n,m,!}q_{n,m}^{-1}f_{n,!}\DR_{\LT,n}^B\\
    &\ov{(**)}{=}\dlim_{m\ge n}i_{n,!}q_{n,m,!}f_{m,!}\pi_{n,m}^{-1}\DR_{\LT,n}^B\\
    &=\dlim_{m\ge n} i_{m,!}f_{m,!}\pi_{n,m}^{-1}\DR_{\LT,n}^B
\end{align*}
where $(*)$ follows from \cite[(5.3.8)]{heyer20246functorformalismssmoothrepresentations}, and $(**)$ follows from the prim base change theorem. Therefore,
\begin{align*}
    \dlim_n q_n^{-1}f_{n,!}\DR_{\LT,n}^B&=\dlim_n\dlim_{m\ge n}i_{m,!}f_{m,!}\pi_{n,m}^{-1}\DR_{\LT,n}^B\\
    &=\dlim_{n\le m}\dlim_m i_{m,!}f_{m,!}\pi_{n,m}^{-1}\DR_{\LT,n}^B\\
    &=\dlim_m i_{m,!}f_{m,!}\DR_{\LT,m}^B.
\end{align*}
Since $i_{m,!}f_{m,!}\DR_{\LT,m}^B$ computes the compactly supported de Rham cohomology of $\calM_{\LT,m}^B$, and we know the cohomology of rigid analytic varieties is the same as the associated Berkovich spaces by \cite[1.6]{Berkovich}. From this we deduce
\begin{align*}
    f_!\DR_{\LT}^B\isom \dlim_n \Gamma_{\dR,c}(\calM_{\LT,n}).
\end{align*}
Applying the similar argument to the Drinfeld side, we see
\begin{align*}
    f_!\DR_{\Dr}^B\isom \dlim_n \Gamma_{\dR,c}(\calM_{\Dr,n}).
\end{align*}
Consequently, 
\begin{align*}
    \dlim_n \Gamma_{\dR,c}(\calM_{\LT,n})\ov{G\times \check G}{\isom}\dlim_n \Gamma_{\dR,c}(\calM_{\Dr,n})
\end{align*}
where the $G\times \check{G}$-equivariancy is clear from the construction.
\end{proof}

\bibliographystyle{amsalpha}

\bibliography{DrLT.bib}
\end{document}